\documentclass{amsart}

\usepackage{graphicx}
\usepackage{xcolor}
\usepackage[T2A]{fontenc}
\usepackage[utf8]{inputenc}
\usepackage[russian,english]{babel}

\newtheorem{thm}{Theorem}[section]
\newtheorem{cor}[thm]{Corollary}
\newtheorem{lem}[thm]{Lemma}
\newtheorem{prop}[thm]{Proposition}

\newtheorem{exa}[thm]{Example}
\newtheorem{dfn}[thm]{Definition}
\newtheorem{rem}[thm]{Remark}
\numberwithin{equation}{section}

\begin{document}

\title[]{3-braids with 4-valent vertices and knot invariant}

\author{S. Kim} 
\address{Jilin University, Changchun, China}
\email{\rm kimseongjeong@jlu.edu.cn}

\author{V.O.Manturov}
\address{Moscow Institute of Physics and Technology}
\email{\rm vomanturov@yandex.ru}


\begin{abstract}

We construct a picture-valued invariant of classical knots by using the group $G_n^3$ and the plat closure of braids. We introduce a modification of the group $G_n^3$ and define a map from the braid group to framed 6-valent graphs with leaves. Each 6-valent vertex corresponds to the moment when three points become collinear. By using this construction, we obtain a knot invariant valued in equivalence classes of graphs modulo local moves. Our invariant provides a new graphical approach to the study of classical knots.
\end{abstract}

\maketitle


\section{Introduction}

Picture-valued invariants play a great role in low-dimensional topology.
Indeed, having a numerical (polynomial, group etc.) invariant of
a topological object (say, knot, manifold), we can judge
about properties of an object indirectly: estimate its number of crossings,
genera and other characteristics. However, having an invariant
valued in linear combinations of graphs (pictures), we can judge about
the geometry of all diagrams of this objects. It is like an image recognition:
rather than having numerical characteristics we can judge about images and
patterns inside diagrams.
This method proved to be extremely effective in virtual knot theory \cite{MFKN, Manturov}.
In the present paper, we construct a picture-valued invariant of knots.
It turns out that the methods do not work directly in classical knot theory,
hence, we want to construct a map from ``the classical world'' (where invariants are usual) to the virtual world'' (where invariants are powerful).
Such a map is constructed explicitly for braids \cite{Man-class-vir}.
For virtual braids which appear as images of classical braids, we get
graphs (braid diagrams) enjoying lots of properties of picture-valued
invariants. When passing to knots, we discuss ways of passing from braids
to knots (mostly, plat closure)
and construct invariants of classical knots valued in some
graphs modulo certain moves.
 
The paper is organised as follows.
In section 2, we discuss basic notions for the theory of groups $G_{n}^{3}$, their generalisations, and the corresponding graph theory, called 3-braids.
In section 3, we construct the closure of 3-braids corresponding to the plat closure of braids on $2n$ strands. It is proved that the closure of 3-braids is (almost) invariant under Birman's moves for the plat closure. In section 4, the graph theory for $G_{n}^{3}$ is generalised as a double cover over 3-braids and anologous theorems are proved. In Section 5, we refer our further research orientation to represent the classical crossing in the 3-braids coming from classical braids.

\subsection{Acknowledgements} 
We are deeply grateful to I. N. Nikonov and Y. Liu for their valuable discussions. The author is supported by RSCF 25-21-00884 ``Applied combinatorial geometry and topology''.

\section{Basic notions of 3-braids}

\begin{dfn}
Let $G$ be a graph. We allow multi-edges and circle components. A vertex $v$ of $G$ with $2n$ incident edges is called {\bf framed} if semi-edges incident to this vertex are divided into $n$ pairs of {\em formally opposite}. A multigraph $G$ is called {\bf framed} if every vertex is framed with $2k$ incident edge for some $k \in \mathbb{N}$ or is a leaf.
\end{dfn}

\begin{dfn}
An {\em unicursal component} of a framed graph is an equivalence class of edges of the framed graph under the following equivalence relation: two edges $e$ and $e'$ are equivalent if there is a sequence of edges $\{e_{i}\}_{i=1}^{n}$ such that $e_{1}=e, e_{n}=e'$ and $e_{i}$ and $e_{i+1}$ are opposite for all $i=1,\dots, n-1$.
\end{dfn}

\begin{dfn}
A multigraph $G$ is called {\bf $*$-graph} if for each vertex $v$ the incident edges are ordered.
\end{dfn}

\begin{dfn}
{\em A 3-braid diagram on $n$ strands} is a framed $*$-graph with $2n$ leaves such that it can be placed in the square $[0,1]\times [0,1]$ satisfying following conditions:
\begin{enumerate}
\item all vertices except leaves have 6 incident edges;
\item $n$ leaves are placed on $[0,1] \times \{0\}$ and remaining $n$ points are on $[0,1] \times \{1\}$;
\item vertices can be placed for edges to go down strictly from $[0,1] \times \{1\}$ to $[0,1] \times \{0\}$.
\end{enumerate}
\end{dfn}
\begin{rem}
When drawing a diagram on a plane we always assume the ordering to be inherited from the plane: the leftmost component is the first, the middle one goes after it, and the rightmost one is the last).
\end{rem}
\begin{dfn}
	A {\em 3-braid on $n$ strand} is an equivalence class of 3-braid diagrams modulo the three 3-braid moves described in Fig.~\ref{fig:3braidmoves}.
	\begin{figure}
\centering\includegraphics[width=300pt]{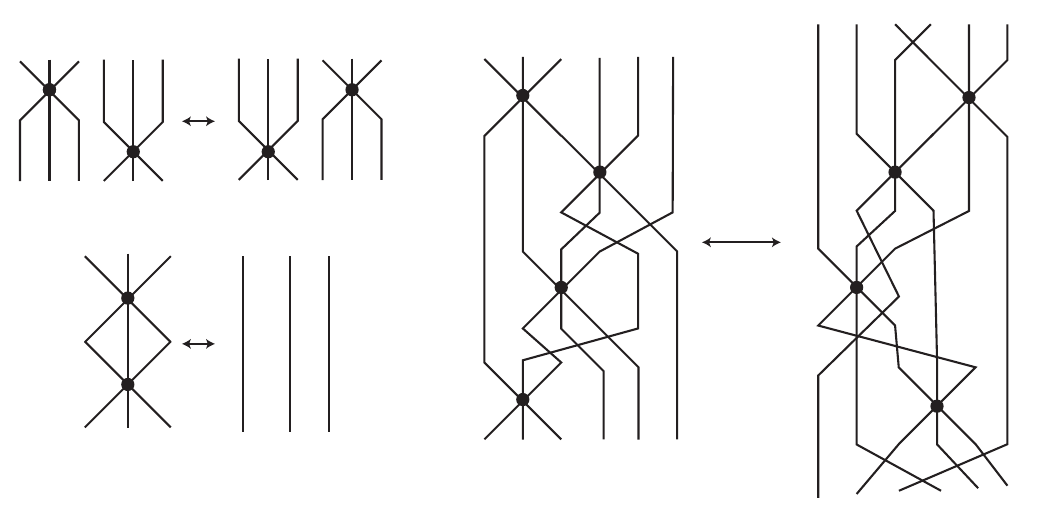}
\caption{A 3-free moves}
\label{fig:3braidmoves}
\end{figure}
\end{dfn}

\begin{dfn}
Let $\mathbf{A}_{n} = \{a_{(i,j,k)} | i,j,k \in \bar{n}\}$. The group $\hat{G}_{n}^{3}$ is defined by
$$\hat{G}_{n}^{3} = \langle \mathbf{A}_{n} | R1,R2,R3\rangle,$$
where 
\begin{itemize}
\item R1: $a_{ijk}a_{kji} =1$,
\item R2: $a_{ijk}a_{stu} = a_{stu}a_{ijk}$ \,\,\, if $|\{i,j,k\} \cap \{s,t,u\}|<2$,
\item R3: $a_{ijk}a_{ijl}a_{ikl}a_{jkl} = a_{jkl}a_{ikl}a_{ijl}a_{ijk}$.
\end{itemize}
\end{dfn}

{\bf Construction of the map from braid to $\hat{G}_{n}^{3}$}

Let $\beta$ be a braid on $n$ strand. Then $\beta$ corresponds to an equivalence class $[\gamma]\in \pi_{1}(C_{n}(\mathbb{R}^{2}), P= (P_{1}, \dots, P_{n}))$, where $(P_{1}, \dots, P_{n})$ are placed on the upper half circle. For the closed curve $\gamma \in \Omega(C_{n}(\mathbb{R}^{2}))$, one can find finitely many moments $0< t_{0} <\dots <t_{s}<1$ such that at $t=t_{l}$ three points $P_{i},P_{j}, P_{k}$ are on one straight line, say $P_{j}$ is placed between $P_{i}$ and $P_{k}$. Without out loss of generality, we may assume that at each moment there is exactly one triple of points on one line. If two ordered vectors $(\overrightarrow{P_{i}P_{j}}, \overrightarrow{P_{i}P_{k}})$ ($(\overrightarrow{P_{k}P_{j}}, \overrightarrow{P_{k}P_{i}})$, resp.) agree on the orientation of $\mathbb{R}^{2}$, then we give a generator $a_{ijk}$ ($a_{kji}$,  resp.) of $\hat{G}_{n}^{3}$, see Fig.~\ref{fig:maps-to-Gn3}.

\begin{figure}
\centering\includegraphics[width=100pt]{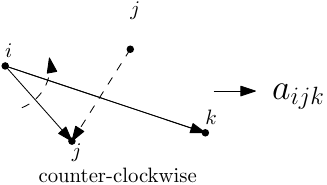}

\caption{ }
\label{fig:maps-to-Gn3}
\end{figure}

 Writing generators of $\hat{G}_{n}^{3}$ in order corresponding to each moment $t_{l}$ we obtain a word in $\hat{G}_{n}^{3}$.

\begin{exa}
From the given 3 strand braid in Fig.~\ref{Borromean_to_Gn3} one can obtain an element $a_{321}a_{312}a_{132}a_{123}a_{213}a_{231} \in \hat{G}_{3}^{3}$.
\begin{figure}
\centering\includegraphics[width=180pt]{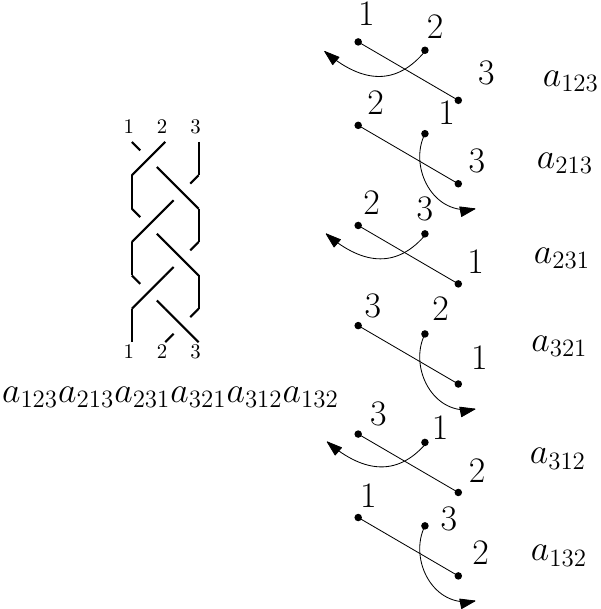}

\caption{ A braid on 3 strands and its corresponding 3-braid. We obtain $a_{321}a_{312}a_{132}a_{123}a_{213}a_{231} \in \hat{G}_{3}^{3}$.}
\label{Borromean_to_Gn3}
\end{figure}
\end{exa}
\begin{lem}\label{lem:braid-Gn3}
The map from n strand braid group to $\hat{G}_{n}^{3}$ defined as above is well-defined.
\end{lem}

{\bf Maps from $\hat{G}_{n}^{3}$ to 3-braids on $\binom{n}{2}$ strands}

Let us define a map from $\hat{G}_{n}^{3}$ to 3-braids on $\binom{n}{2}$ strands as follows: we associate each generator $a_{ijk}$ to a framed 6-valent vertex so that the strand $(ik)$ passes through the vertex in the middle, the strand $(ij)$ passes through the vertex from left upside to right downside and the strand $(jk)$ passes through the vertex from right upside to left downside as described in Fig.~\ref{fig:a_ijk_two_triples-1}.

\begin{figure}
\centering\includegraphics[width=250pt]{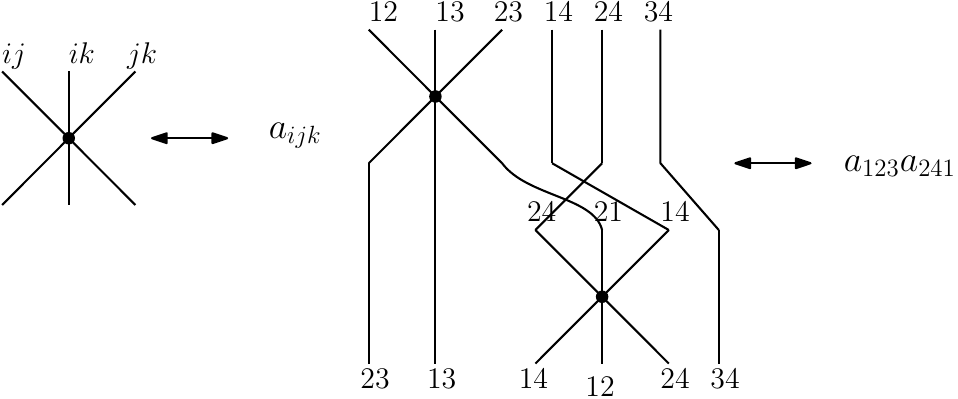}

\caption{ $a_{ijk}$ and triple points of a 3-braid}
\label{fig:a_ijk_two_triples-1}
\end{figure}

For example, one can obtain $\binom{3}{2}$ strand 3-braid diagram from an element of $\hat{G}_{n}^{3}$ as in Fig.~\ref{Borromean_to_3-free}. We assume that each strand of the corresponding 3-braid diagram is labeled by $(ij)$.

\begin{figure}
\centering\includegraphics[width=180pt]{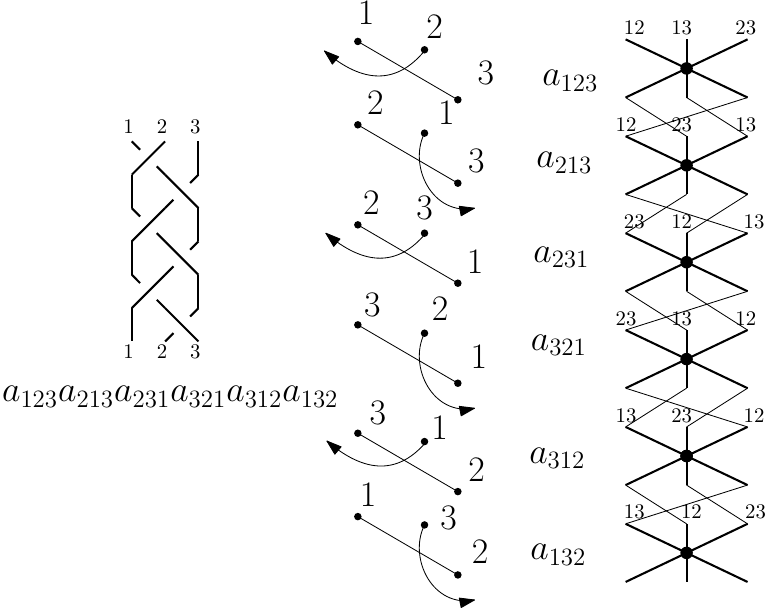}

\caption{ A braid on 3 strands and its corresponding 3-braid. We obtain $a_{321}a_{312}a_{132}a_{123}a_{213}a_{231} \in \hat{G}_{3}^{3}$.}
\label{Borromean_to_3-free}
\end{figure}

\begin{cor}\label{lem:Gn3-3braids} 
The map from $\hat{G}_{3}^{3}$ to 3-braids on $\binom{n}{2}$ strands defined as above is well-defined.
\end{cor}

\begin{thm}
    There is a map from $B_{n}$ to 3-braids on $n\choose 2$ strands. We denote it by $T(\beta)$ for $\beta \in B_{n}$.
\end{thm}

\section{framed 6-valent braids with 4-valent vertices}

Generally,  for a given knot or link ``reading'' a diagram along a certain family of straight lines "writing out" the moment when three points lie on a line give a 3-diagram. But, the problem is, this is sometimes, not well-defined. 

More precisely, the problem takes place when two maxima (or two minima) of our knot or link interchange their height with respect to the direction of the ``reading''. Obviously, this does not change our knot or link, but in the level of 3-diagrams, "the saddle move'' arises, as described in Fig.~\ref{fig:max_saddle}, and apparently, it is not moves for 3-free links. 
\begin{figure}
\centering\includegraphics[width=200pt]{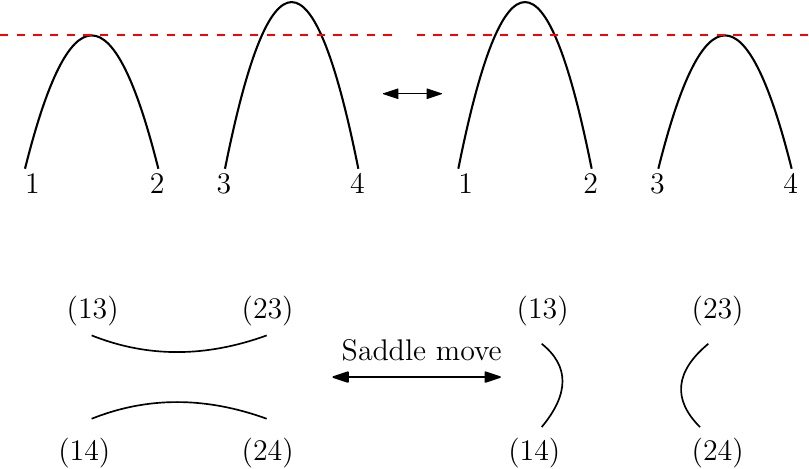}
\caption{Example of mapping for non-uniform braid}
\label{fig:max_saddle}

\end{figure}
Therefore, to construct well-defined mappings, we need to use knot theories where such transformation does not appear. The case of closed non-uniform braids and radial lines is an example of such theory. 

\begin{dfn}
Let $\beta \in B_{2n}$ be a braid on $2n$ strands. {\em A plat closure of $\beta$} is obtained by the plat closure as depicted in~Fig~\ref{plat_link}.
\end{dfn}

\begin{prop}\cite{Birman}
Every link $L$ can be presented as a plat closure of a braid $\beta \in B_{2n}$. 
\end{prop} 
In \cite{Birman} a kind of Markov theorem for plat closure of braids is proved. 

\begin{prop}\cite{Birman}
Let $L_{1}$ and $L_{2}$ be equivalent links. Let $\beta_{1} \in B_{2n_{1}}$ and $\beta_{2}\in B_{2n_{2}}$ be braids such that the plat closure of $\beta_{i}$ is equivalent to $L_{i}$, $i=1,2$. Then there exists $t \geq max(n_{1},n_{2})$ such that for each $n \geq t$ 
$$\beta'_{i} = \beta_{i}\sigma_{2n_{1}}\sigma_{2n_{1}+2}\dots \sigma_{2n-2} \in B_{2n}, i=1,2$$
are equivalent modulo the moves depicted in Fig.~\ref{move_plat}.
\end{prop} 

\begin{figure}
\centering\includegraphics[width=8cm]{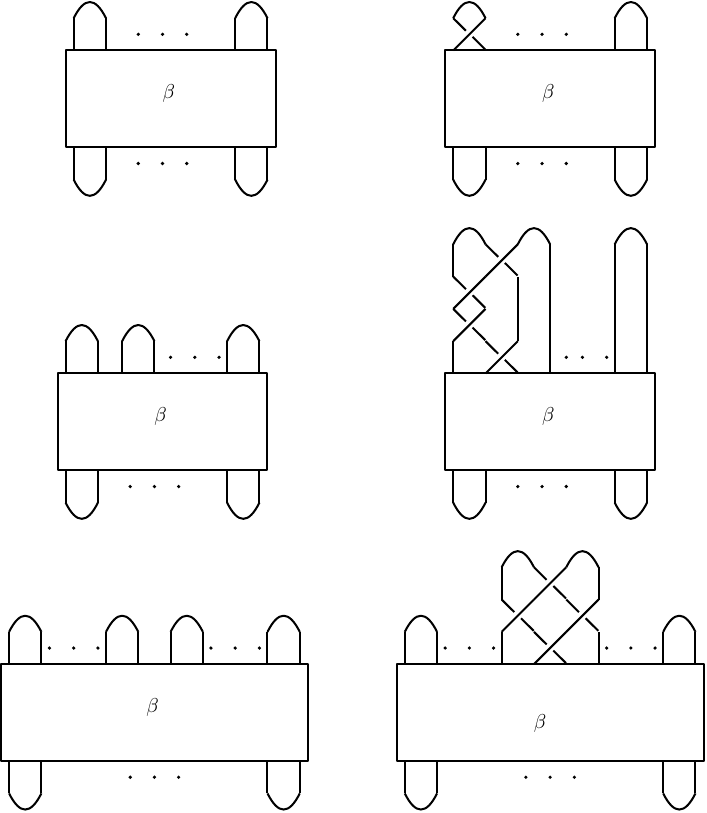}

\caption{Moves for plat closures}
\label{move_plat}
\end{figure}



Let us construct the closure of the 3-braid $T(\beta)$ corresponding to a braid on $2n$ strands: Let a link $L$ be a diagram in the form of plat representation, i.e. it is a plat closure $\bar{\beta}$ of a braid $\beta$ in $Br_{2n}$, see Fig.~\ref{plat_link}. 
\begin{figure}
\centering\includegraphics[width=5cm]{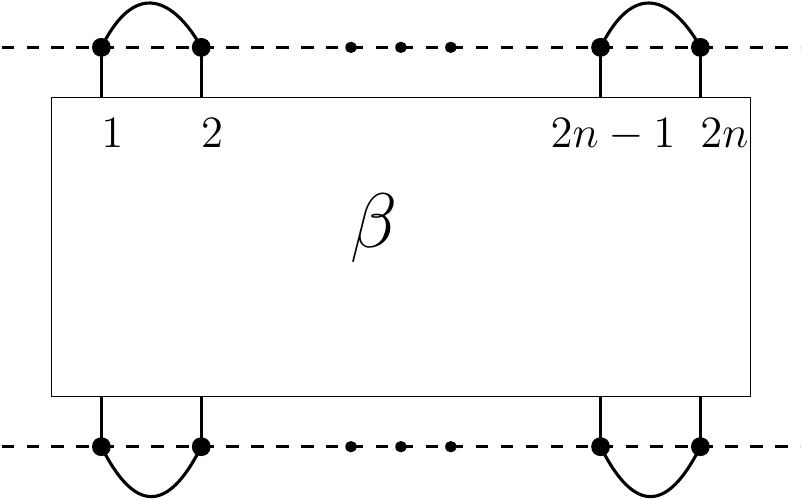}

\caption{Link in the form of plat representation}
\label{plat_link}
\end{figure}
First we give labels for each strands of $\beta$ by $1,\dots, 2n$ as in Fig.~\ref{plat_link} and obtain a 3-braid $T(\beta)$ with ends from the braid $\beta$ as in Section 3. Here each component of the 3-braid with ends is labeled by $(ij)$. Now we close the 3-braid with ends as follow: each maximal (minimal) between strands $i$ and $j$ corresponds to the end points on $[0,1] \times \{1\}$ ($[0,1] \times \{0\}$). For any two maximals of $\bar{\beta}$ which connecting strands $i$ and $j$, and $s$ and $t$, respectively, we connect ends $(is),(it),(js),(st)$ of the obtained 3-diagram to a vertex (but not framed) as described in Fig.~\ref{mapping_to_3link}. 
\begin{figure}
\centering\includegraphics[width=150pt]{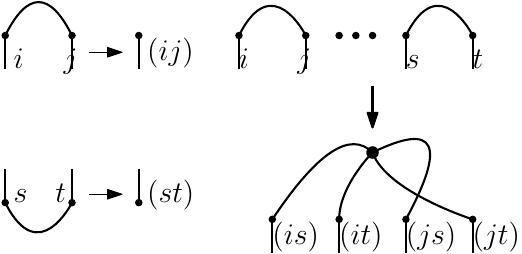}

\caption{Closing of end points of 3-diagram with endpoints}
\label{mapping_to_3link}
\end{figure}

\begin{figure}
\centering\includegraphics[width=8cm]{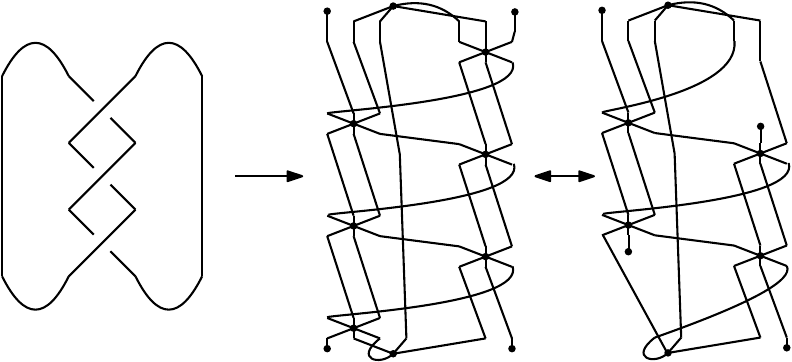}

\caption{The closure of $T(\beta)$ from a knot $\bar{\beta}$}
\label{exa_link_to_3-diag_ends}
\end{figure}

By the construction of 3-diagrams with 4-valent vertices from links in plat representations, we obtain the following corollary.
\begin{cor}
For a link $L=\bar{\beta}$, $\beta \in Br_{2n}$ in plat representations we obtain a 3-diagrams with 4-valent vertices and it is invariant under Artin moves on $\beta$ and under changes of heights of maximals and minimals. Let us denote it by $\overline{T(\beta)}$ or $T(L)$. 
\end{cor}

\begin{proof}
The invariance under Artin moves on $\beta$ is followed from the proof of Lemma~\ref{lem:braid-Gn3} and \ref{lem:Gn3-3braids}. Since we do not consider the heights of maximum and minimum, but consider pairs of maximum and minimum, which correspond to 4-valent vertices, the proof is completed.
\end{proof}


\begin{thm}
Let $\beta$ and $\beta'$ in $B_{2n}$ be braids such that their plat closures give a same link. Then $\overline{T(\beta)}$ and $\overline{T(\beta')}$ are equivalent 3-diagrams with 4-valent vertices up to the move, depicted in Fig.~\ref{fig:pass_endpoint} and~\ref{fig:cancel-triples}. Here notice that 1-valent vertex cannot pass through ``the middle'' of triple points.

\begin{figure}
\centering\includegraphics[width=8cm]{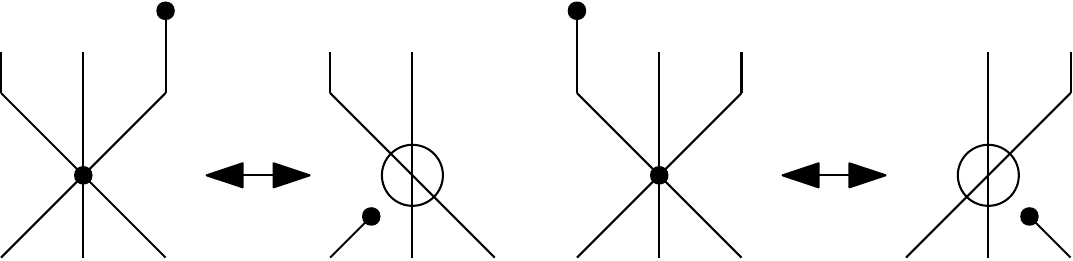}

\caption{Passing triple points of a end point}
\label{fig:pass_endpoint}
\end{figure}

\begin{figure}
\centering\includegraphics[width=8cm]{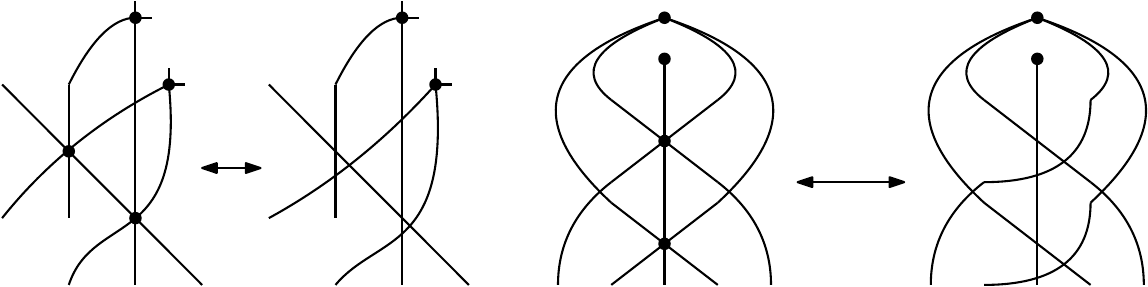}
\caption{Cancellations of two adjacent 6-valent vertices. The first move comes from that two consecutive triple points are $(i,j,k),(i,j,l)$ and $k,l$ make a maximum. The right move comes from that two consecutive triple points are $(i,j,k),(i,l,k)$ and $i,k$ and $j,l$ make maximums.}
\label{fig:cancel-triples}
\end{figure}
\end{thm}

\begin{proof}
It suffices to show that if $\bar{\beta'}$ is obtained from $\bar{\beta}$ by one of moves in Fig.~\ref{move_plat}, then $\overline{T(\beta')}$ and $\overline{T(\beta)}$ are equivalent 3-diagram with 4-valent vertices up to the move in Fig.~\ref{fig:pass_endpoint} and~\ref{fig:cancel-triples}. 

Assume that $\beta'=\sigma_{1}\beta$ are links in the left and the right parts of Fig.~\ref{proof_plat_move1}. Then $\overline{T(\beta')}$ locally has a diagram as the second figure in Fig.~\ref{proof_plat_move1}, which contains arcs labeled by $(13),(14), (23),(24)$. By using the move in Fig.~\ref{fig:pass_endpoint}, we can remove triple points, which are obtained from $\sigma_{1}$. Analogously, we can remove triple points with arcs labeled by $(1(2k+1)),(1(2k+2)), (2(2k+1)),(2(2k+2))$.

Assume that $\beta'$ is obtained from $\beta$ by the second move in Fig.~\ref{move_plat}, that is, $\beta'= \sigma_{2}\sigma_{1}^{2}\sigma_{2}\beta$. Locally $\overline{T(\beta')}$ has a diagram as the second figure in Fig.~\ref{proof_plat_move3}, which contains arcs labeled by $(12),(13),(14)$, $(23),(24),(34)$ and we can remove all 6-valent vertices as described in Fig.~\ref{proof_plat_move3}.

On the other hand, each pair $2k+1,2k+2$ for $k>1$, we obtain pairs of consecutive triple points $\{(2,3,2k+1),(2,3,2k+2)\}$, $\{(1,3,2k+1),(1,3,2k+2)\}$, $\{(2,3,2k+1),(2,3,2k+2)\}$ and $\{(1,4,2k+1),(1,4,2k+2)\}$. Since $2k+1$ and $2k+2$ make a maximum, $(2(2k+1))$ and $(2(2k+2)$ are attached to a 4-valent vertex, and so are $(3(2k+1))$ and $(3(2k+2)$. Therefore, 6-valent vertices corresponding to $\{(2,3,2k+1),(2,3,2k+2)\}$ are removed by the move in Fig.~\ref{fig:cancel-triples}. Analogously, remaining pairs of consecutive triple points are removed.

Assume that $\beta'$ is obtained from $\beta$ by the third move in Fig.~\ref{move_plat}. For simplicity, we say, $\beta' = \sigma_{2}\sigma_{1}\sigma_{3}\sigma_{4}$. Then $\overline{T(\beta')}$ locally has a diagram as the second figure in Fig.~\ref{proof_plat_move2}, which contains arcs labeled by $(12),(13),(14), (23),(24),(34)$. As the case of the second move, triple points consisting of edges with labels $(1(2k+1))$, $(2(2k+1))$, $(3(2k+1))$, $(4(2k+1))$, for $k>1$, can be removed. 

\begin{figure}
\centering\includegraphics[width=10cm]{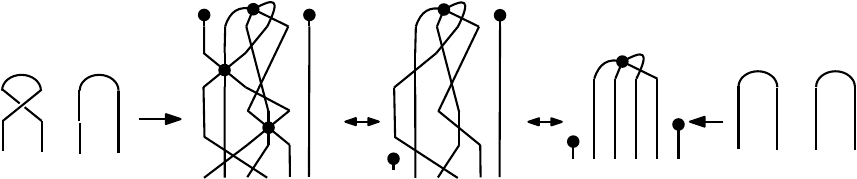}

\caption{The first move for plat closure and corresponding 3-diagrams with 4-valent vertices}
\label{proof_plat_move1}
\end{figure}

\begin{figure}
\centering\includegraphics[width=10cm]{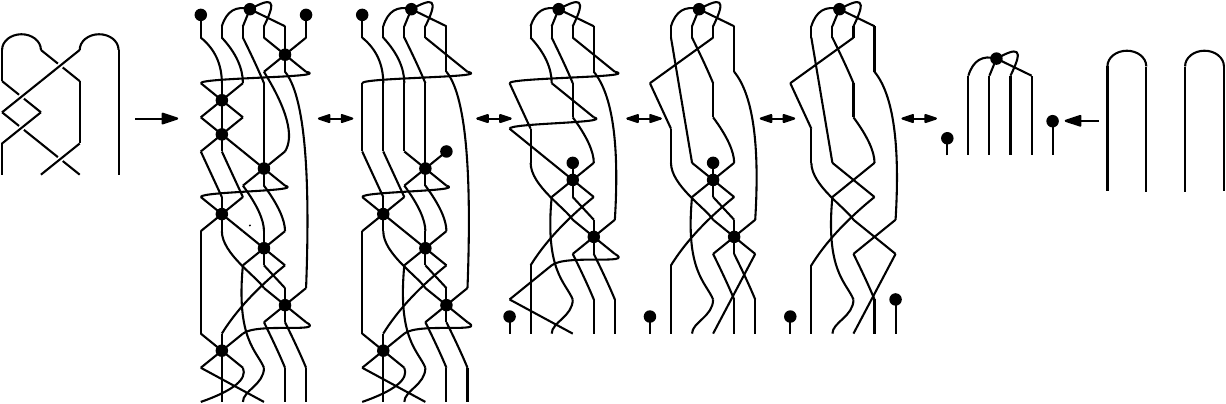}

\caption{The second move for plat closure and corresponding 3-diagrams with 4-valent vertices}
\label{proof_plat_move3}
\end{figure}

\begin{figure}
\centering\includegraphics[width=10cm]{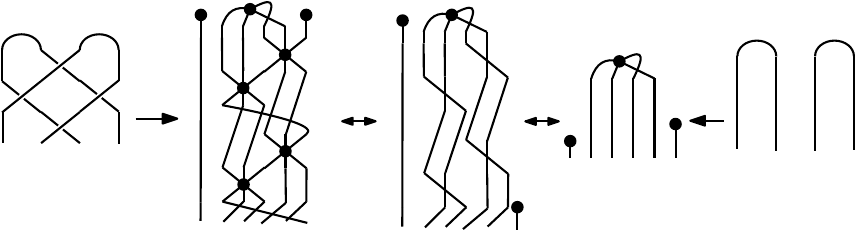}

\caption{The third move for plat closure and corresponding 3-diagrams with 4-valent vertices}
\label{proof_plat_move2}
\end{figure}
\end{proof}

\begin{figure}
\centering\includegraphics[width=8cm]{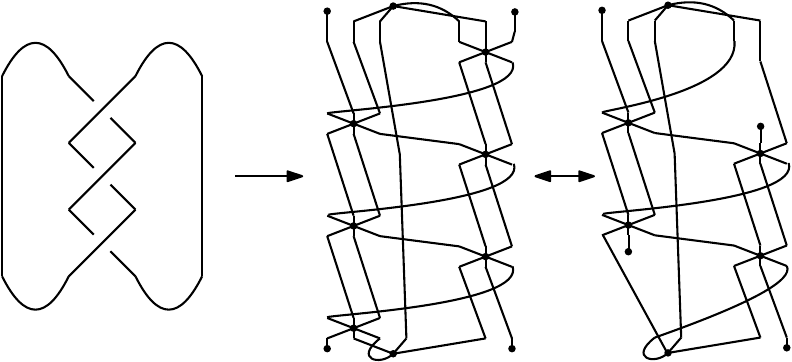}

\caption{Corresponding element $a_{234}a_{231}a_{324}a_{321}a_{234}a_{231}$ in $\hat{G_{4}^{3}}$ and the corresponding 3-diagram with 4-valent vertices to trefoil}
\label{fig:exa-link-3diagram-trefoil}
\end{figure}

\begin{lem}\label{lem:inv-stabil}
Let $\beta \in B_{2n}$ and $\beta' \in B_{2n+2}$. Let $G$ and $G'$ be corresponding 3-braids to $\beta$ and $\beta'$ respectively. If $\beta' = \beta \sigma_{2n}$, then $\overline{T(\beta')}$ is equivalent to $\overline{T(\beta)D}$ by the moves in Fig~\ref{fig:pass_endpoint} and~\ref{fig:cancel-triples}, where $D$ is a braid only with 4-valent vertices.
\end{lem}

\begin{proof}
    Let $\beta' = \beta \sigma_{2n}$, $T(\beta')$ can be separated to two parts: triple points coming from $\beta \cup (\text{2-parallel strands})$ and triple points coming from $\sigma_{2n}$.
    First, let us consider triple points coming from $\beta \cup (\text{2-parallel strands})$. The corresponding 3-braid has the form of $$A_{1}a_{i_{1}j_{1}(2n+1)}a_{i_{1}j_{1}(2n+2)}A_{2}a_{i_{2}j_{2}(2n+1)}a_{i_{2}j_{2}(2n+2)}\dots A_{k}a_{i_{k}j_{k}(2n+1)}a_{i_{k}j_{k}(2n+2)}A_{k+1},$$ where $A_{1}A_{2}\dots A_{k}A_{k+1}$ is the 3-braid corresponding to $\beta$. By the move in Fig.~\ref{fig:cancel-triples}, we can remove all additional pairs $a_{i_{s}j_{s}(2n+1)}a_{i_{s}j_{s}(2n+2)}$. Then triple points coming from $\sigma_{2n}$ has the form of $a_{2n,2n+1,2n+2}a_{2n,2n+1,1}a_{2n,2n+1,2}\dots a_{2n,2n+1,2n-1}$. One can easily see that the triple points corresponding to $a_{2n,2n+1,2n+2}$ and $a_{2n,2n+1,2n-1}$ can be removed by the move in Fig.~\ref{fig:pass_endpoint}. Since $$a_{2n,2n+1,1}a_{2n,2n+1,2}\dots a_{2n,2n+1,2n-3}a_{2n,2n+1,2n-2}$$ can be paired by $$(a_{2n,2n+1,1},a_{2n,2n+1,2}),(a_{2n,2n+1,3},a_{2n,2n+1,4}), \dots,(a_{2n,2n+1,2n-3}, a_{2n,2n+1,2n-2}),$$ by the move in Fig.~\ref{fig:cancel-triples}, we can remove them. Note that now strands of $\overline{T(\beta')}$ corresponding to $(i,2n+1)$, $(j, 2n+2)$ and $(2n+1,2n+2)$ do not pass triple points, but connected with 4-valent vertices. By moving them below of 3-braids, we can obtain that $\overline{T(\beta')} = \overline{T(\beta)D}$, where $D$ is a braid only with 4-valent vertices.
\end{proof}

\begin{rem}
$\overline{T(\beta)}$ is not a knot invariant yet, since the moves in Fig~\ref{fig:pass_endpoint} and~\ref{fig:cancel-triples} cannot reduce the number of 4-valent vertices. For example, see Fig.~\ref{fig:exa-3diagram-trefoil-stabilization-simple}. The knot $K$ obtained from the trefoil with one stabilization gives $T(K)$ on the second diagram in Fig.~\ref{fig:exa-3diagram-trefoil-stabilization-simple}. The part of $T(K)$ with labels $(i5)$ or $(i6)$ are colored in red. By applying moves in Fig~\ref{fig:pass_endpoint} and~\ref{fig:cancel-triples} we obtain the right most diagram and it has the diagram connecting $T(3_{1})$ and a graph only with 1 and 4 valent vertices in the blue box. We expect that there is a pattern of connections of 4-valent vertices in $D$ and we can find an appropriate move.
\end{rem}
\begin{figure}
\centering\includegraphics[width=10cm]{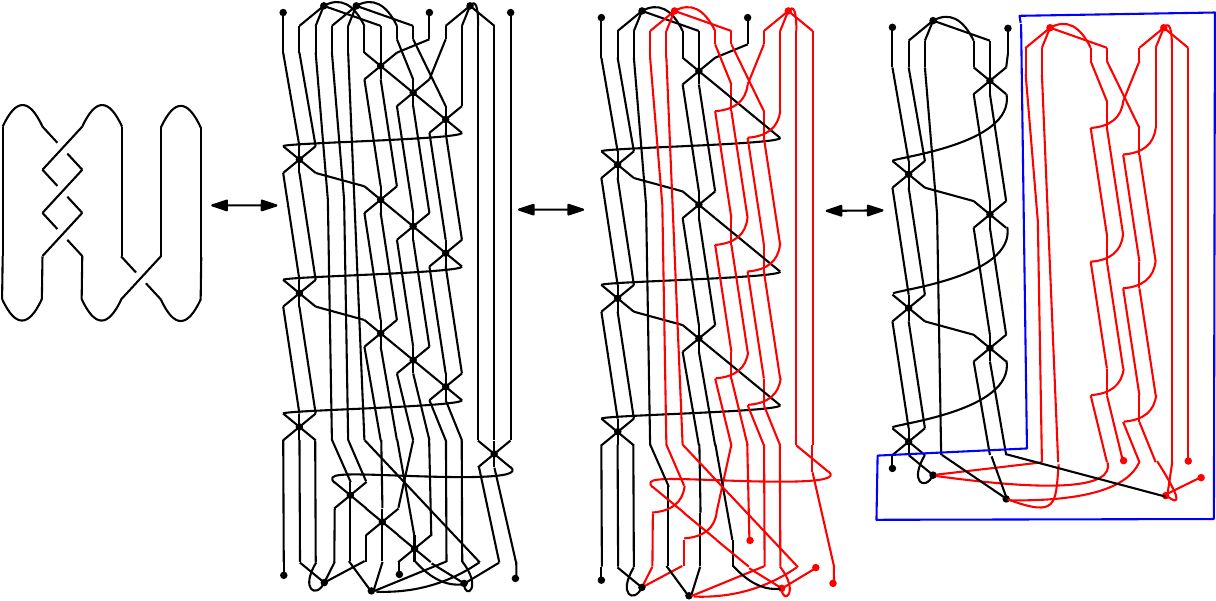}

\caption{The knot $K$ obtained from the trefoil with one stabilization and $T(K)$. }
\label{fig:exa-3diagram-trefoil-stabilization-simple}
\end{figure}
\begin{thm}\label{thm:KUO}
Let $K$ be a link such that $K=\overline{\beta}$ for a braid $\beta$ in $B_{2n}$. Let $\beta' \in B_{2n+2}$ be the image of $\beta$ by the natural inclusion $B_{2n} \hookrightarrow B_{2n+2}$, that is, $\hat {\beta'} =K\sqcup O$. Then
$\overline{T(\beta')} = \overline{T(\beta)} \sqcup \overline{T}$, for a graph $\overline{T}$ consisting of $1,4$-valent vertices.
\end{thm}

\begin{proof}
    As we mentioned in the proof Lemma~\ref{lem:inv-stabil}, corresponding 3-braid to $\beta \cup (\text{2-parallel strands})$ has the form of $$A_{1}a_{i_{1}j_{1}(2n+1)}a_{i_{1}j_{1}(2n+2)}A_{2}a_{i_{2}j_{2}(2n+1)}a_{i_{2}j_{2}(2n+2)}\dots A_{k}a_{i_{k}j_{k}(2n+1)}a_{i_{k}j_{k}(2n+2)}A_{k+1},$$ where $A_{1}A_{2}\dots A_{k}A_{k+1}$ is the 3-braid corresponding to $\beta$. Then $a_{i_{1}j_{1}(2n+1)}a_{i_{1}j_{1}(2n+2)}$ can be removed by the move in Fig.~\ref{fig:cancel-triples}. Then there is no triple points containing edges with labels $(i,2n+1)$, $(j,2n+2)$ and $(2n+1,2n+2)$. Therefore, $T(\beta')=T(\beta)\sqcup T$, where $T$ consists of straight edges with labels $(i,2n+1)$, $(j,2n+2)$ and $(2n+1,2n+2)$. When we construct $\overline{T(\beta')}$ edges in $T(\beta)$ and $T$ are not connected by 4-valent vertices. It completes the proof.
\end{proof}

\begin{exa}\label{exa:KUO}
Let $K$ be a trefoil and $\beta \in B_{4}$ a braid such that $\overline{\beta} =K$. Then $\beta' \in B_{6}$ such that $\hat{\beta'} = K\sqcup O$ corresponds to a closure $\overline{T(\beta'})$ of a 3-braid and, as described in Fig.~\ref{fig:exa-3diagram-trefoil0}, it is equivalent to $\overline{T(\beta})\sqcup D$ for a graph $D$ without 6-valent vertices. See Fig.~\ref{fig:exa-3diagram-trefoil0}.
    \begin{figure}
\centering\includegraphics[width=10cm]{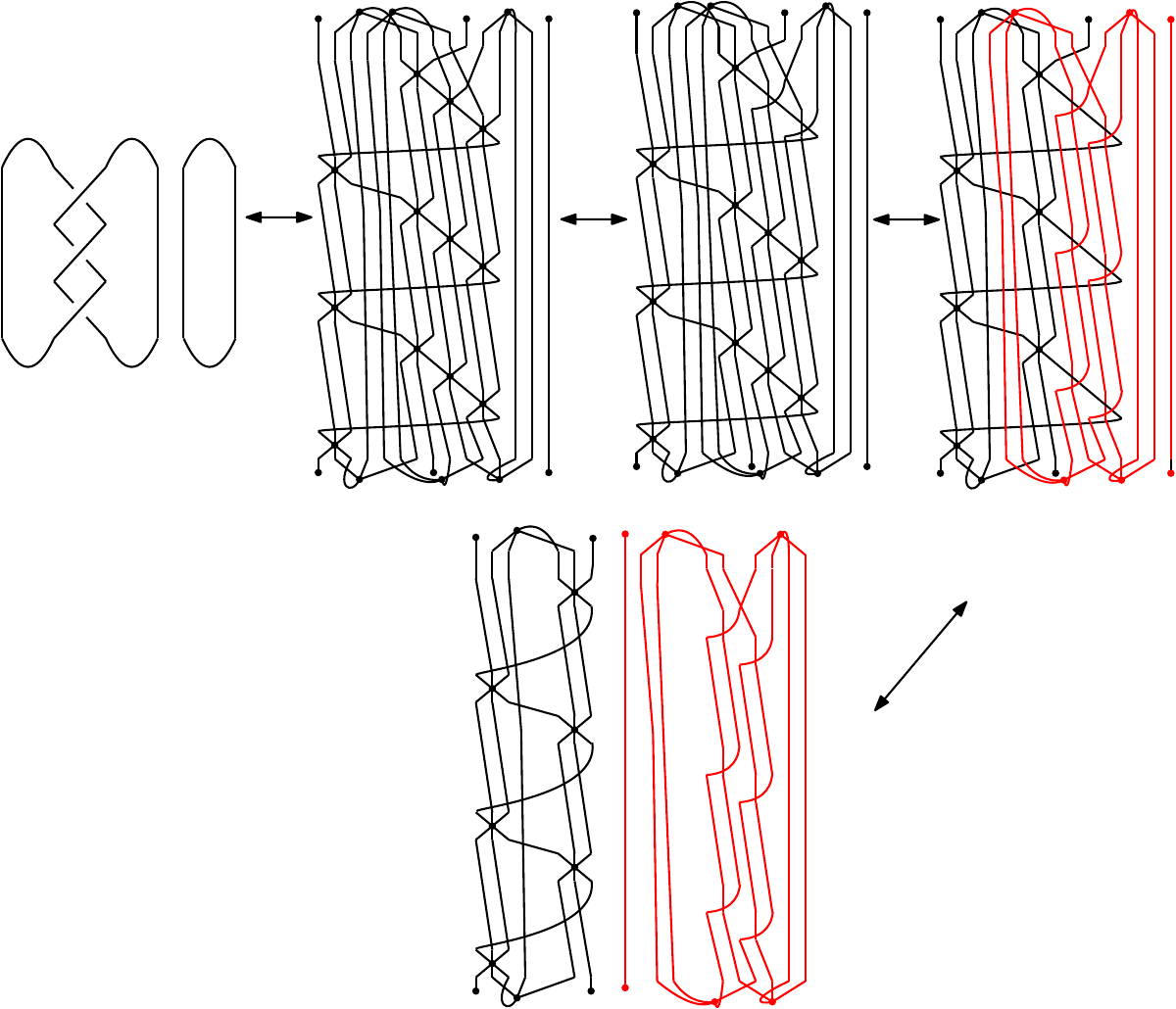}

\caption{For a braid $\beta \in B_{4}$, $\overline{T(\beta')}$ contains $\overline{T(\beta)}$.}
\label{fig:exa-3diagram-trefoil0}
\end{figure}
\end{exa}

\section{Generalization of $T(\beta)$ -- double cover $DT(\beta)$ over $T(\beta)$}

We can generalize the graph $T(\beta)$ with the ordered labels: $(ij) \neq (ji)$. We associate each generator of $\hat{G}_{n}^{3}$ to two 6-valent vertices as described in Fig.~\ref{fig:a_ijk-double-triples}. Then we obtain 3-braids on $n(n-1)$ strands from $w \in \hat{G}_{n}^{3}$.

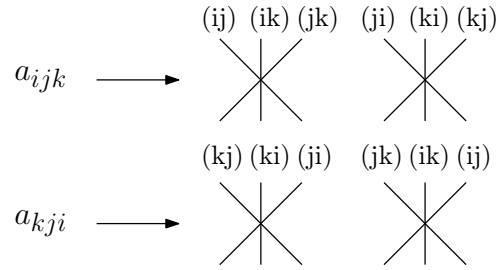
\begin{figure}
\centering\includegraphics[width=6cm]{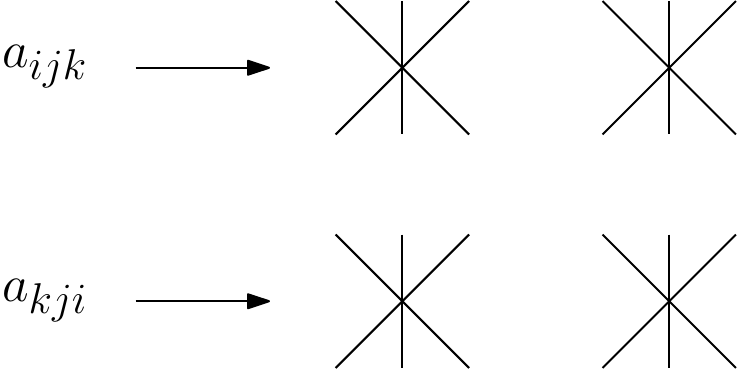}
\put(-100,90){(ij)}
\put(-82,90){(ik)}
\put(-64,90){(jk)}
\put(-40,90){(ji)}
\put(-22,90){(ki)}
\put(-4,90){(kj)}
\put(-100,38){(kj)}
\put(-82,38){(ki)}
\put(-64,38){(ji)}
\put(-40,38){(jk)}
\put(-22,38){(ik)}
\put(-4,38){(ij)}
\caption{Each generator $a_{ijk}$ of $\hat{G}_{n}^{3}$ is associated with a pair of 6-valent vertices.}
\label{fig:a_ijk-double-triples}
\end{figure}

\begin{thm}
    There is a map from $B_{n}$ to 3-braids on $n(n-1)$ strands. We denote it by $DT(\beta)$ for $\beta \in B_{n}$.
\end{thm}

Now we define the closure of $DT(\beta)$ in analogous way to the closure of $T(\beta)$. See Fig.~\ref{fig:mapping_to_double3link}. Let us denote the closure of $DT(\beta)$ by $\overline{DT(\beta)}$.
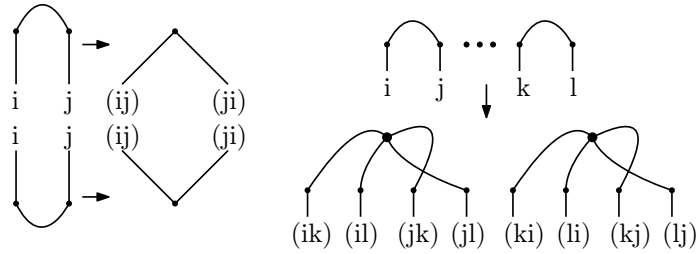
\begin{figure}
\centering\includegraphics[width=250pt]{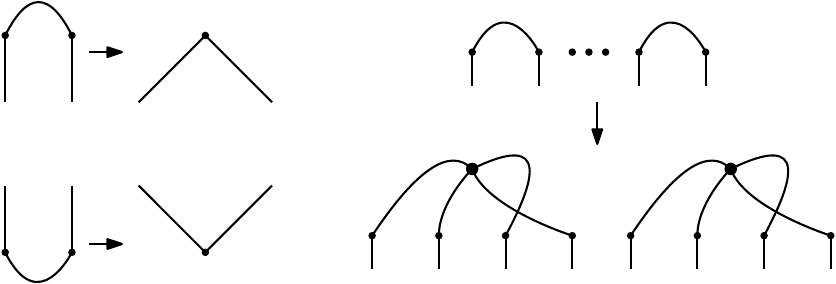}
\put(-250,45){i}
\put(-230,45){j}
\put(-215,45){(ij)}
\put(-175,45){(ji)}
\put(-250,32){i}
\put(-230,32){j}
\put(-215,32){(ij)}
\put(-175,32){(ji)}
\put(-110,50){i}
\put(-90,50){j}
\put(-60,50){k}
\put(-40,50){l}
\put(-145,-5){(ik)}
\put(-125,-5){(il)}
\put(-105,-5){(jk)}
\put(-85,-5){(jl)}
\put(-65,-5){(ki)}
\put(-45,-5){(li)}
\put(-25,-5){(kj)}
\put(-5,-5){(lj)}
\caption{Closing of end points of $DT(\beta)$}
\label{fig:mapping_to_double3link}
\end{figure}
The moves in Fig.~\ref{fig:pass_endpoint} and~\ref{fig:cancel-triples} are associated with moves in Fig.~\ref{fig:cancel-type1-trivial} and~\ref{fig:cancel-double-triples} for $\overline{DT(\beta)}$. Then analogously we can prove the following statements.

\begin{figure}
\centering\includegraphics[width=250pt]{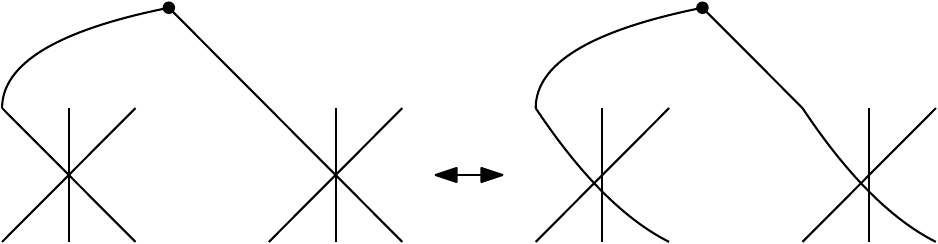}

\caption{Cancellations of two 6-valent vertices I}
\label{fig:cancel-type1-trivial}
\end{figure}

\begin{figure}
\centering\includegraphics[width=250pt]{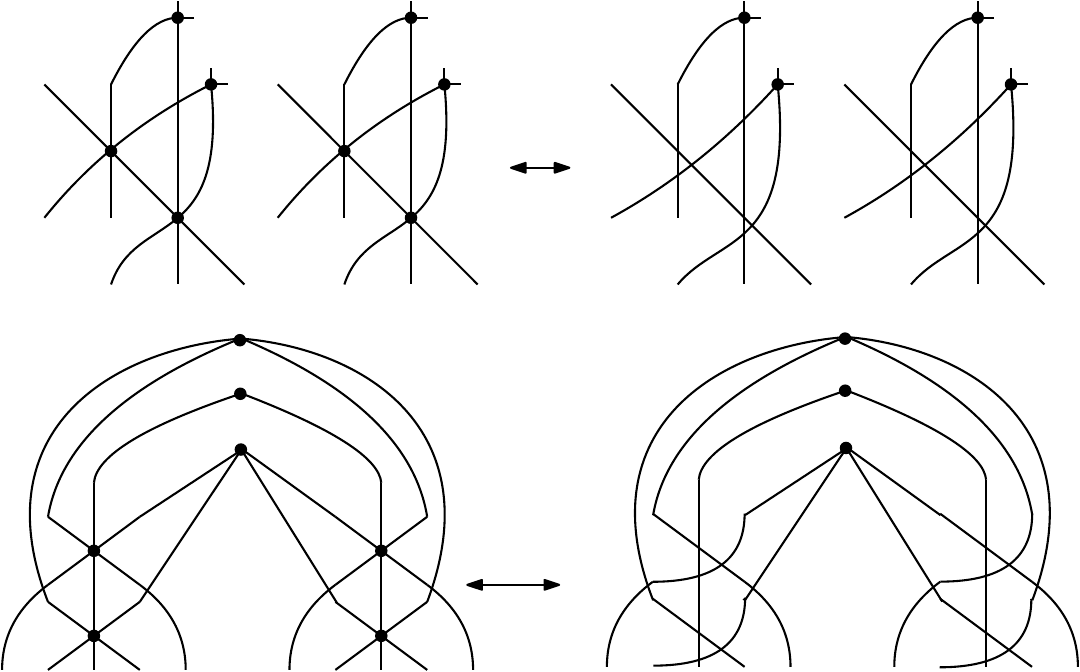}
\caption{Cancellations of two 6-valent vertices II. The first move comes from that two consecutive triple points are $(i,j,k),(i,j,l)$ and $k,l$ make a maximum. The right move comes from that two consecutive triple points are $(i,j,k),(i,l,k)$ and $i,k$ and $j,l$ make maximums.}
\label{fig:cancel-double-triples}
\end{figure}

\begin{thm}
Let $\beta$ and $\beta'$ in $B_{2n}$ be braids such that their plat closures give a same link. Then $\overline{DT(\beta)}$ and $\overline{DT(\beta')}$ are equivalent 3-diagrams with 4-valent vertices up to the move, depicted in Fig.~\ref{fig:cancel-type1-trivial} and~\ref{fig:cancel-double-triples}. Here notice that 1-valent vertex cannot pass through ``the middle'' of triple points.
\end{thm}
\begin{lem}\label{lem:inv-stabil-double}
Let $\beta \in B_{2n}$ and $\beta' \in B_{2n+2}$. Let $G$ and $G'$ be corresponding 3-braids to $\beta$ and $\beta'$ respectively. If $\beta' = \beta \sigma_{2n}$, then $\overline{DT(\beta')}$ is equivalent to $\overline{DT(\beta)D}$ by the moves in Fig.~\ref{fig:cancel-type1-trivial} and~\ref{fig:cancel-double-triples}, where $D$ is a braid only with 4-valent vertices.
\end{lem}

\section{Further research: classical crossings of $T(\beta)$}

In the previous sections we constructed 3-braids $T(\beta)$ and $DT(\beta)$, and their closure corresponding to classical links. In particular, each strand corresponds to the straight line connecting two strands of a link (a secant) and each 6-valent vertex corresponds to the singular moment, three points are placed on one straight line, or equivalently, three secants are collinear. In the present construction $T(\beta)$ and $DT(\beta)$ forget classical crossings. We can overcome such a weakness with simple modification by considering classical links with a separately placed vertical line, numbered by $\infty$, see Fig.~\ref{fig:stereographic-proj-trefoil}.

\begin{figure}
\centering\includegraphics[width=5cm]{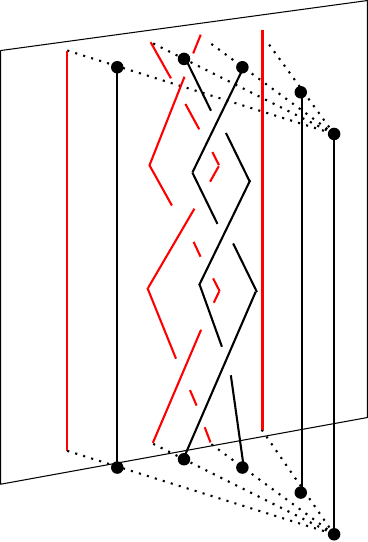}
\put(-12,160){$\infty$}
\caption{A diagram of a knot can be obtained by a projection from $\infty$-vertical line to a plane.}
\label{fig:stereographic-proj-trefoil}
\end{figure}

If we follow the construction of $T(\beta)$, then it contains strands with labels $(i\infty)$. Moreover, an intersection between strands labeled by $(i\infty)$ and $(j\infty)$ corresponds to a classical crossing of $\beta$ between strands $i$ and $j$. Such a deduction follows from the fact that a knot diagram is the shadow of the projection of knots with respect to rays coming from a point, which is placed far from the knot. Simply speaking the model with one vertical link ``contains'' the classical approach with crossings and Reidemeister moves.

The above idea can be applied in the more compatible way to the construction of $\overline{T}(\beta)$ by putting a trivial knot $O$ separately from the given link $L = \overline{\beta}$. Then the secants connecting strands of $L$ and $O$ plays a role of $(i\infty)$. Our expectation can be verified with the example~\ref{exa:KUO}. As we proved in Theorem~\ref{exa:KUO}, $T(K\cup O) = T(K) \sqcup T$, where $K$ is the trefoil knot. In the example~\ref{exa:KUO}, each strands of $T$ in red color has a label $(i5)$ or $(i6)$, see Fig.~\ref{fig:exa-3diagram-trefoil0-trefoil}. We can obtain a diagram of the trefoil knot without over/under information by identifying $(i5)$ and $(i6)$ as described in Fig.~\ref{fig:exa-3diagram-trefoil0-trefoil}. We expect that our graphs $T(\beta)$ and $DT(\beta)$ give different information with knot diagrams for classical crossings.
 \begin{figure}
\centering\includegraphics[width=10cm]{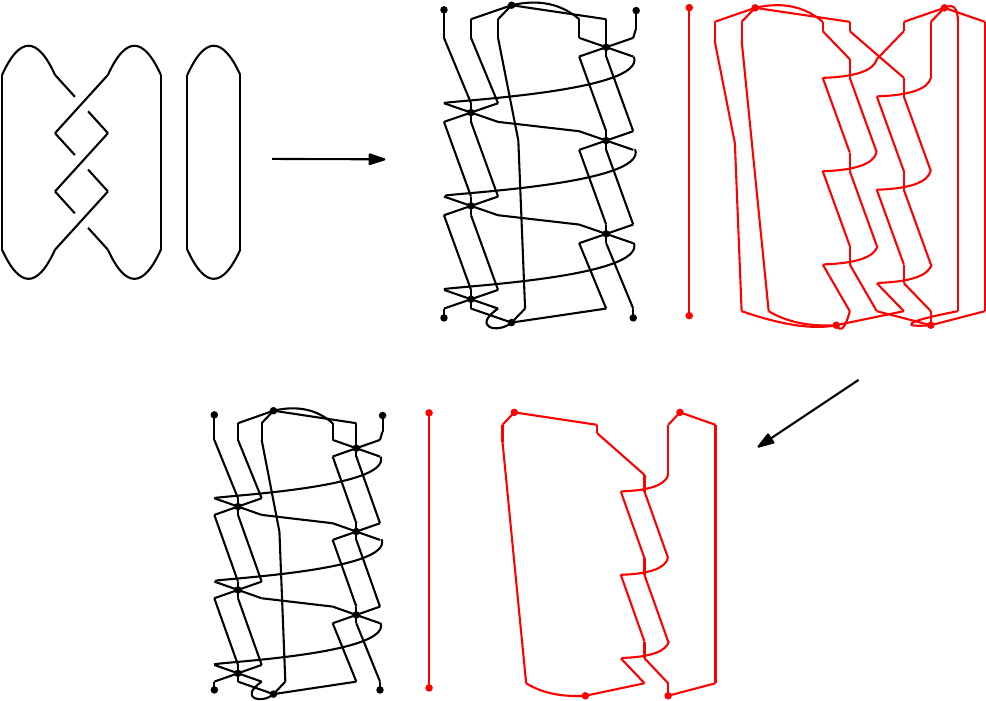}
\put(-90,190){\footnotesize{$(15)$}}
\put(-75,180){\footnotesize{$(16)$}}
\put(-70,205){\footnotesize{$(25)$}}
\put(-50,200){\footnotesize{$(26)$}}
\put(-30,200){\footnotesize{$(35)$}}
\put(-25,190){\footnotesize{$(36)$}}
\put(-15,180){\footnotesize{$(45)$}}
\put(-5,190){\footnotesize{$(46)$}}
\put(-140,95){$T(3_{1})$}
\put(-40,95){$T$}
\caption{$T(3_{1} \sqcup O) = T(3_{1})\sqcup T$. In particular, by identifying strands with $(i5)$ and $(i6)$ in $T$ we obtain a knot diagram of $3_{1}$ without over/under information.}
\label{fig:exa-3diagram-trefoil0-trefoil}
\end{figure}
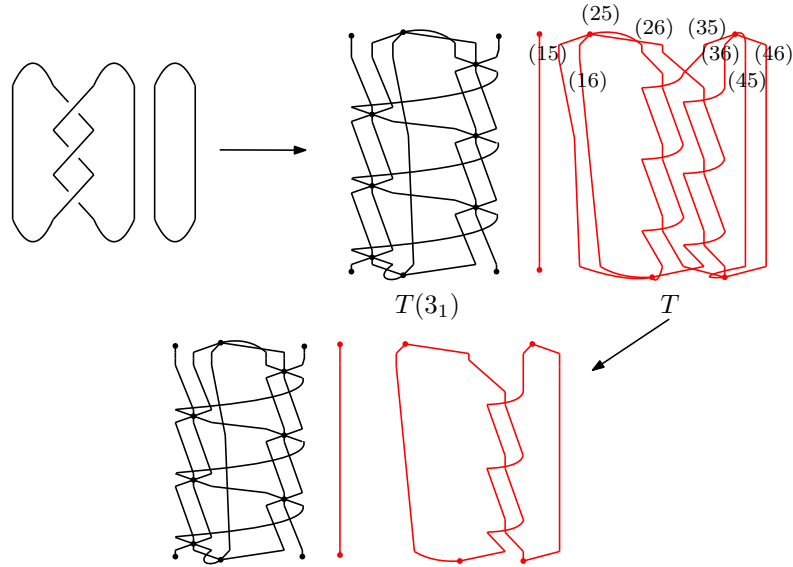
We will continue to study such graphs containing classical crossings with over/under information.

\end{document}